\documentclass[reqno,11pt]{amsart}

\usepackage{amssymb, amsthm, latexsym, amscd, amsfonts, textcomp}
\usepackage[all, cmtip]{xy}
\usepackage{mdwlist}
\usepackage{graphicx}

\newcommand{\iso}{\cong}
\newcommand{\Hom}{\mbox{Hom}}

\newcommand{\git}{/\!\!/}

\newcommand{\CC}{\mathbb{C}}

\newcommand{\ZZ}{\mathbb{Z}}
\newcommand{\LL}{\mathbb{L}}

\DeclareMathOperator{\R}{R}
\DeclareMathOperator{\Tr}{Tr}
\DeclareMathOperator{\GL}{GL}
\DeclareMathOperator{\DT}{DT}
\DeclareMathOperator{\Hilb}{Hilb}
\newcommand{\vir}{\textrm{vir}}

\theoremstyle{plain} 

\newtheorem{example}[equation]{Example}
\newtheorem{lemma}{Lemma}[section]
\newtheorem{prop}{Proposition}[section]
\newtheorem{theorem}{Theorem}[section]

\newtheorem{D}{Definition}[section]

\title{Motivic Invariants of Quivers via Dimensional Reduction}

\author[Andrew Morrison]{Andrew Morrison}
\address{Dept. of Math., University of British Columbia, Vancouver, BC, Canada.}
\email{andrewmo@math.ubc.ca}

\date{March 2011}

\begin{document}

\maketitle

\begin{abstract}
We provide a reduction formula for the motivic Donaldson-Thomas 
invariants associated to a quiver with superpotential. The method 
is valid provided the superpotential has a linear factor, it allows 
us to compute virtual motives in terms of ordinary motivic classes of 
simpler quiver varieties. We outline an application, giving explicit 
formulas for the motivic Donaldson-Thomas invariants of the orbifolds $[\CC \times \CC^2/\ZZ_n]$.
\end{abstract}

\section{Introduction to Dimensional Reduction.}\label{intro}

In \cite{thom} Thomas defines integer valued invariants associated to the moduli spaces of sheaves on a Calabi-Yau threefold. 
These numbers contain geometric information about the underlying manifold. In particular they provide a virtual count for the 
number of curves in each homology class, conjecturally equivalent \cite{MNOP} to the invariants of Gromov-Witten. 

\paragraph{}Recently there have been several extensions of the integer valued Donaldson-Thomas invariant \cite{mot3CY}, \cite{COHA}, \cite{BBS}, \cite{JS}. 
In \cite{mot3CY} Kontsevich and Soibelman propose to work in a general three dimensional Calabi-Yau category. 
For a choice of stability condition they get moduli spaces of objects in this category each of which has a motivic DT invariant. 
The Euler numbers of these motives specialize to the classical invariant. 

\paragraph{} Quivers with superpotential provide concrete examples of such CY3 categories. 
Although quiver categories are of independent interest in representation theory \cite{rein} they often contain geometric content, 
in particular the derived category of many local CY3 folds can be realized in this way.

\paragraph{}We provide a reduction theorem for motivic DT invariants of a general class of quivers with superpotential. 
Our theorem expresses the motivic DT invariants in terms of the ordinary motivic classes of certain reduced quiver representations.

\paragraph{}In this way our work can be seen as a generalization of \cite{BBS} where the quiver with superpotential models the Hilbert scheme of $\CC^3$. 
Or from another point of view, in the cohomological Hall algebra, the reduction statement of \cite{COHA} section 4.8 
is the analog of our motivic result. 
\paragraph{}Our main application of this reduction theorem is to compute motivic DT invariants of orbifolds which specialize to the 
known classical invariants \cite{young}, a forthcoming paper \cite{a_n}.

\paragraph{}Let $Q$ be a finite quiver with vertex set $I$, that is a finite directed graph. The representations of $Q$ are indexed by a dimension vector 
$\mathbf{v} \in \mathbb{Z}_{\geq 0}^I$. At each vertex $i$ we have a vector space $V_i$ of dimension $v_i$, and for each arrow $\alpha : i \to j $ 
we have a linear map $M_{\alpha}: V_i \to V_j$. Two such representations are equivalent up to a change of basis vectors at each vertex. We produce a 
moduli space of quiver representations by taking a G.I.T quotient. The stability condition comes from the extra data of a framing vector 
$\mathbf{f}\in \ZZ_{\geq 0}^I \setminus \{0\}$, together with a choice of a linearization of the group action $\chi$;
\[ N^Q(\mathbf{v} , \mathbf{f}) = \prod_{\alpha : i \to j} \Hom(V_i,V_j) \times \prod_{i\in I} (V_i)^{f_i} \git _\chi \prod_{i\in I} \GL(V_i). \]
It is a smooth quasi-projective variety (see section 2).
\paragraph{}The path algebra $\CC Q$ of the quiver $Q$, is the vector space over $\CC$ with basis given by the paths in the quiver, the multiplication is then
 defined by concatenation of paths. An element $W\in \CC Q$ of the path algebra is called a superpotential if it is the sum of cycles, i.e. paths that
 form loops. We are specifically interested in superpotentials that have a linear factor;
\[ W = L \cdot R \] 
here $L$ is linear, that is a sum of length one paths.

\paragraph{}The superpotential defines a function on the moduli space of representations of $Q$. Given a representation $(M_\alpha)_{\alpha : i\to j}$ insert 
the matrix $M_\alpha$ in place of each occurrence of $\alpha$ in $W$ then take the trace;
\[ \Tr W : N^Q(\mathbf{v},\mathbf{f}) \to \CC \]
the definition is invariant under a change of basis for the representation and so gives a well defined regular function. We define the DT moduli space of 
$Q,W$ to be the scheme theoretic degeneracy locus of $\Tr W$;
\[\DT^{Q,W}(\mathbf{v},\mathbf{f}) = \{ d\Tr W = 0 \} \subset N^{Q,W}(\mathbf{v},\mathbf{f}). \]
Behrend, Bryan and Szendr\H{o}i \cite{BBS} define a virtual motive associated to each such a locus, essentially given by the motivic Milnor fibre of the map $\Tr W$. 
Using the same definition we have a motivic DT invariant;
\[ [\DT^{Q,W}(\mathbf{v},\mathbf{f})]_{\textrm{vir}} \in K_0(Var_\CC)[\LL^{-\frac{1}{2}}]\]
taking values in the Grothendieck ring of varieties, adjoined with a formal inverse square root of the Lefschetz motive $\LL = [\mathbb{A}^1_{\CC}]$. 
\paragraph{}When the superpotential has a linear factor $W=L\cdot R$, there exists a reduced space of quiver representations;
\[R^{Q,W}(\mathbf{v}) = \{ (M_\alpha)_{\alpha : i\to j} \mid R((M_\alpha)_{\alpha : i\to j}) = 0 \} \subset \prod_{\alpha : i\to j} \Hom(V_i,V_j). \]
Here the space is a variety and also we have no G.I.T quotient. Our result expresses the virtual motives of the DT moduli spaces in terms of 
the ordinary motives of the reduced spaces. It is best phrased in the context of a quantized Poisson algebra. 

\paragraph{}Let $\mathcal{M}_\CC^{St}$ be the Grothendieck ring of varieties where we formally invert the general linear groups and a square root of the 
Lefschetz motive. Now consider the formal power series over this ring, as a $\mathcal{M}_\CC^{St}$-module it is defined
\[ S_Q := \left\{ \sum_{ \mathbf{v} \in \mathbb{Z}^I_{\geq 0}} m_{\mathbf{v}} \mathbf{t}^{\mathbf{v}} \mid m_\mathbf{v} \in \mathcal{M}_\CC^{St} \right\} .\]
The Euler form $\langle , \rangle_Q$, a pairing associated to the adjacency matrix of $Q$, defines a non-commutative product $\ast$ on this set given by,
\[ \mathbf{t}^\mathbf{v} \ast \mathbf{t}^\mathbf{w} = \LL^{-\langle \mathbf{w},\mathbf{v} \rangle_Q} \mathbf{t}^{\mathbf{v} + \mathbf{w}}. \]
$S_Q$ is the quantized Poisson algebra of the quiver $Q$. The virtual motives of the DT moduli spaces give an element;
\[ \DT^{Q,W}_{\mathbf{f}}(\mathbf{t}) := \sum_{\mathbf{v}\in \ZZ^{I}_{\geq 0}} \LL^{-\frac{1}{2}\langle \mathbf{v}, \mathbf{v} \rangle_{Q}}
[\DT^{Q,W}(\mathbf{v},\mathbf{f})]_{\vir}\mathbf{t}^{\mathbf{v}} \in S_Q. \]
Similarly the reduced spaces can be assembled into a series;
\[\textrm{R}^{Q,W}(\mathbf{t}) := \sum_{\mathbf{v}\in \ZZ^{I}_{\geq 0}} \frac{[R^{Q,W}(\mathbf{v})]}{[\GL(\mathbf{v})]}\mathbf{t}^{\mathbf{v}} \in S_Q \]
here $\GL(\mathbf{v}) =\prod_{i\in I}\GL_{v_i} $. In this context our main result (Theorem 7.1) reads simply,
 
\[ \boxed{ \textrm{R}^{Q,W}(\LL^{\frac{\mathbf{f}}{2}}\mathbf{t}) = \DT^{Q,W}_\mathbf{f}(\mathbf{t})
\ast \textrm{R}^{Q,W}(\LL^{-\frac{\mathbf{f}}{2}}\mathbf{t}) }  \] 
where $(\LL^{\mathbf{a}}\mathbf{t})^\mathbf{v} = \LL^{\mathbf{a}\cdot \mathbf{v}}\mathbf{t}^{\mathbf{v}}$. This describes all the motivic DT invariants in terms of the ordinary motivic classes 
of the reduced spaces. In many cases computing the motives of the reduced spaces can be carried out explicitly, often leading to simple product formulas. 
In a future paper \cite{a_n} we will use this method to compute motivic DT invariants of Calabi-Yau orbifolds.

\section{Quiver with Superpotential.}\label{secqui}

Let $Q$ be a finite quiver with vertex set $I$ and arrows $\alpha : i \to j$. A path in the quiver is a sequence of arrows 
$\alpha_1 \cdot \alpha_2 \cdots \alpha_n$ so that the head of the arrow $\alpha_i$ coincides with the tail of the arrow 
$\alpha_{i+1}$, each path has a length given by the number of arrows of which it is composed.
\paragraph{}The path algebra of the quiver $\CC Q$, is the $\CC$ vector space with basis the set of all paths, and with multiplication given 
by concatenation of paths. An element $W \in \mathbb{C}Q$ 
is called a superpotential if all the monomials in $W$ are cycles. Then
\[ \tilde{W} \in \mathbb{C}Q / [\mathbb{C}Q,\mathbb{C}Q] \]
is the class of $W$ up to equivalence of cyclically permuting terms in any monomial. Ginzberg \cite{ginz},
provides a detailed account of such algebras and superpotentials. Our main interest is in their representations. 
For each dimension vector $\mathbf{v} \in \mathbb{Z}_{\geq 0}^I$ we define a linear space of representations of $Q$,
\[ M^{Q}(\mathbf{v}) = \prod_{\alpha : i \to j } \Hom(\mathbb{C}^{v_i},\mathbb{C}^{v_j}) \]
the group of automorphisms $\GL (\mathbf{v}) = \prod_{i\in I}\GL(v_i)$ acts on this space by change of basis.
The superpotential $W$ defines a function on the representations of $Q$. So that given a representation $(M_\alpha)_{\alpha : i \to j}$, 
we define $W((M_\alpha)_{\alpha:i \to j})$ to be the linear map obtained by inserting $M_\alpha$ in the place of each occurrence
 of $\alpha$ in $W$. Taking the trace of this linear map gives a complex number $\Tr W((M_\alpha)_{\alpha:i \to j})$
 well defined under cyclic reordering of each monomial and invariant under the action of $\GL(\mathbf{v})$. 
In particular this means that the class
 $\tilde{W}\in \mathbb{C}Q / [\mathbb{C}Q,\mathbb{C}Q]$ defines a $\GL (\mathbf{v})$ equivariant Chern-Simons functional,
\[ \Tr \tilde{W}_{\mathbf{v}} : M^Q(\mathbf{v}) \to \mathbb{C}. \]
The superpotentials we consider in this paper satisfy the following special condition.
\begin{D} 
A superpotential $\tilde{W} \in \mathbb{C}Q / [\mathbb{C}Q,\mathbb{C}Q]$ has a linear factor if for 
some lift $W\in \mathbb{C}Q$ there exists a factorization 
\[ W = L\cdot R \in \mathbb{C}Q \]
where $L$ is a linear combination of arrows, so that for each pair of vertices $i,j\in I$
 there is at most one such arrow $\alpha : i \to j$ between them. Moreover we have that no arrow in $L$ occurs in $R$. 
\end{D}

For such a superpotential we define
\begin{eqnarray*}
M_1^{Q,W}(\mathbf{v})&:=& \Tr\tilde{W}_{\mathbf{v}}^{-1}(1), \\
M_0^{Q,W}(\mathbf{v})&:=& \Tr\tilde{W}_{\mathbf{v}}^{-1}(0), \\
R^{Q,W}(\mathbf{v})  &:=& \{ (M_\alpha)_{\alpha:i \to j}\in M^Q(\mathbf{v})\mid R((M_\alpha)_{\alpha:i \to j}) = 0 \}
\end{eqnarray*}

the first two being fibres of the Chern-Simons functional and the latter the zero locus of the reduced equations $R = 0$. 
The corresponding moduli stacks are defined,
\begin{eqnarray*}
\mathfrak{M}_1^{Q,W}(\mathbf{v}) &:=& [M_1^{Q,W}(\mathbf{v}) / \GL (\mathbf{v}) ], \\
\mathfrak{M}_0^{Q,W}(\mathbf{v}) &:=& [M_0^{Q,W}(\mathbf{v}) / \GL (\mathbf{v}) ], \\
\mathfrak{R}^{Q,W}(\mathbf{v})   &:=& [R^{Q,W}(\mathbf{v})   / \GL (\mathbf{v}) ].
\end{eqnarray*}
As a final piece of notation we introduce two pairings on the lattice $\ZZ_{\geq 0}^I$ to be used later. 
Let $\mathbf{v},\mathbf{w}\in \ZZ^I$ then 
\begin{eqnarray*}
 && \mathbf{v} \cdot \mathbf{w}  = \sum_{i\in I} v_iw_i \\
 && \langle \mathbf{v},\mathbf{w} \rangle_Q = \sum_{i\in I} v_iw_i - \sum_{\alpha:i\to j} v_iw_j.
\end{eqnarray*}

\section{Stability for Quiver Representations.}\label{secstab}

We fix a dimension vector $\mathbf{v}\in \ZZ^I_{\geq 0}$. The stability condition depends upon a choice of framing vector
$\mathbf{f}\in \ZZ^I_{\geq 0} \setminus \{0\}$, introducing 
such will produce a fine moduli space of quiver representations. Let
\[ U^Q(\mathbf{v},\mathbf{f}) = \{ (M_\alpha, m_l) \mid \underline{\dim} \mathbb{C}\langle M_\alpha \rangle \{ m_l \} = \mathbf{v} \}
 \subset M^Q(\mathbf{v}) \times \prod_{i\in I}(\CC^{v_i})^{f_i} \]
where $\mathbb{C}\langle M_\alpha \rangle \{ m_l\}$, 
is defined to be the span of the collection of vectors $m_l$ under the successive action of the matrices $M_\alpha$. 
Now the group $\GL (\mathbf{v})$ acts freely on the $U^Q(\mathbf{v},\mathbf{f})$ giving a fine moduli space of quiver representations.

\begin{prop}{\cite[Szendr\H{o}i]{SzendroiNCDT}, \cite[King]{king}}\label{quivermodulispace}
There exists a character $\chi : \GL(\mathbf{v})\to \CC^*$ such that the open subset 
$U^Q(\mathbf{v},\mathbf{f})$ is precisely the subset of stable points for the action of $\GL(\mathbf{v})$ linearized by $\chi$.
In particular the action of $\GL(\mathbf{v})$ on $U^Q(\mathbf{v},\mathbf{f})$ is free, and the quotient
\[ N^Q(\mathbf{v},\mathbf{f}) := M^Q(\mathbf{v})\times \prod_{i\in I}(\CC^{v_i})^{f_i} \git _\chi \GL (\mathbf{v}) 
= U^Q(\mathbf{v},\mathbf{f})/\GL (\mathbf{v}) \]
is a smooth quasi-projective G.I.T. quotient.
\end{prop}
The Chern-Simons functional is $\GL(\mathbf{v})$ equivariant and consequently descends to a regular function on the smooth quotient,
\[ \Tr \tilde{W}_{\mathbf{v}} : N^Q(\mathbf{v},\mathbf{f}) \to \CC.  \]
Our main object of interest is the scheme theoretic degeneracy locus of this functional, which we define;
\[ \DT^{Q,W}(\mathbf{v},\mathbf{f}) := \{ d \Tr \tilde{W}_{\mathbf{v}} = 0 \} \subset  N^Q(\mathbf{v},\mathbf{f}) . \]

\begin{example}\label{HilbnC3example}
Let $Q$ be the quiver with one vertex $\{\star\}$ and three arrows $\{x,y,z\}$, superpotential $W = x(yz-zy)$ 
then it is well known \cite[Prop 2.1]{BBS},
\[\DT^{Q,W}(n,1) = \Hilb^n(\CC^3). \]
\end{example}

\section{Grothendieck Rings.}

Here is an account of Grothendieck rings following Bridgeland \cite{bridge} (c.f. To\"en \cite{toen}). 
We are interested in spaces that are varieties or stacks over $\CC$. Denote by,
\[ Var_{\mathbb{C}} \subset St_{\mathbb{C}} \]
the inclusion of the category of varieties in the category of Artin stacks. Our first definition is the Grothendieck ring of varieties.
\begin{D}The Grothendieck ring of varieties $K_0(Var_{\mathbb{C}})$, is the free abelian group on isomorphism classes of varieties over $\mathbb{C}$, modulo the scissor relations 
\[ [X] = [Z] + [X \setminus Z] \]
for $Z$ a subvariety of $X$. Multiplication in $K_0(Var_{\mathbb{C}})$ is given by fibre product
\[[X]\cdot[Y] =[X \times_{\mathbb{C}} Y].\]
\end{D} 
For a variety $X$ we will call $[X]$ the motivic class of $X$. For example consider a Zariski fibration $\pi: E\to B$ with fibre $F$, then stratifying the base by trivializing neighborhoods and using the scissor relations we get,
\[ [E] = [F]\cdot [B] \in K_{0}(Var_{\mathbb{C}}). \]
We can use this fact to give a nice formula for the motivic class of the general linear group.
\begin{lemma}
\[ [\GL_{n}] = \prod_{k=0}^{n-1} (\mathbb{L}^n - \mathbb{L}^{k}) \in K_{0}(Var_{\mathbb{C}}). \]
Where $\mathbb{L} = [\mathbb{A}^1_{\mathbb{C}}]$.
\end{lemma}
\begin{proof}
The map
\[ \pi: \GL_{n} \to \mathbb{C}^n \setminus \{ 0 \}, \]
sending a matrix to its first column, is a Zariski fibration with fibre equal to 
\[ \GL_{n-1} \times \mathbb{C}^{n-1} \]
so by what was said above
\[ [\GL_{n}] = (\mathbb{L}^n - 1) \mathbb{L}^{n-1} [\GL_{n-1}], \]
and the result follows by induction.
\end{proof}
Recall that the Grassmannian can be defined as the free quotient of the general linear group by the subgroup of matrices fixing a hyperplane of dimension $k$, so as an immediate corollary of the above lemma we also have a formula for the motivic class of the Grassmannian, 
\[ [Gr(k,n)] = \frac{[\GL_n]}{[\GL_k][\GL_{n-k}][\Hom(k,n-k)]} \in K_0(Var_{\CC}). \]
To define the Grothendieck ring of stacks we need the notion of geometric bijection, we say that a representable morphism
\[f: X\to Y \]
is a geometric bijection if it induces an isomorphism
\[ f: X(\mathbb{C}) \to Y(\mathbb{C}) \]
between the groupoids of $\mathbb{C}$-valued points. In our definition we shall also need to consider stacks with algebraic stabilizers, that is stacks $X$ locally of finite type over $\CC$ such that for all $x\in X(\CC)$, $\textrm{Isom}_\CC(x,x)$ is an affine algebraic group.
\begin{D} The Grothendieck ring of stacks $K_0(St_{\mathbb{C}})$, is the free abelian group of isomorphism classes of stacks of over $\mathbb{C}$, with algebraic stabilizers, modulo the relations 
\begin{enumerate}
\item $[X\sqcup Y] = [X] + [Y]$.
\item $[X] = [Y]$ for every geometric bijection $f: X\to Y$.
\item $[X] = [Y]$ for every pair of Zariski fibrations with the same base and fibre. 
\end{enumerate}
Again multiplication comes from the fibre product.
\end{D}
The inclusion of the category of varieties in the category of stacks gives an obvious homomorphism of rings 
\[ K_0(Var_{\mathbb{C}}) \xrightarrow{\phi} K_0(St_{\mathbb{C}}) , \]
and To\"en \cite[Theorem 3.10]{toen} proves that its extension 
\[ K_0(Var_{\mathbb{C}})[[\GL_n]^{-1} ; n\geq 1] \xrightarrow{\tilde{\phi}} K_0(St_{\CC}) \]
is an isomorphism. This fact is proven by reducing the class of every stack with algebraic stabilizers to a global quotient $[Y/ \GL_n]$, with $Y$ a variety. The result then follows from this lemma.

\begin{lemma}\label{princebundle}
Every principal $\GL_n$ bundle $\pi : Y \to X$ is a Zariski fibration. In particular 
\[ [X] = [Y] /[\GL_n] \in K_0(St_{\mathbb{C}}) \]
\end{lemma}
\begin{proof}
The fibration $\pi: Y\to X$ is a principal $\GL_n$ bundle if its pullback to any scheme $S$ is,
\[ \xymatrix{ & Y \ar[d]^\pi \\ S\ar[r]^f  & X.} \]
So $f^*\pi : f^* Y \to S$ is a principal $\GL_n$ bundle, therefore locally trivial in the \'etale topology. But as the group $\GL_n$ is special \cite{serre} this bundle is a Zariski fibration.
\end{proof}

Since we need them for our definition of virtual motives we hereby define two extended rings of motivic classes,
\[ \mathcal{M}_{\mathbb{C}} := K_0(Var_{\mathbb{C}})[\mathbb{L}^{-\frac{1}{2}}]
 \hspace{0.5cm} \textrm{   and   } \hspace{0.5cm} 
\mathcal{M}^{St}_{\mathbb{C}} := K_0(St_{\mathbb{C}})[\mathbb{L}^{-\frac{1}{2}}]. \]
The topological Euler characteristic of classes in $K_0(Var_{\mathbb{C}})$ extends to $\mathcal{M}_{\mathbb{C}}$ by letting $\chi(\mathbb{L}^{-\frac{1}{2}}) = -1$, giving a ring homomorphism,
\[ \chi : \mathcal{M}_{\mathbb{C}}\to \ZZ. \]

\section{Virtual Motives.}

Let $f:M\to\CC$ be a regular function on the a smooth quasi-projective variety. Denef and Loeser \cite{denefloeser} study the vanishing cycles
 of such a map and define a motivic class of vanishing cycles $[\varphi_f]$, which is cohomologically equivalent to the Milnor fibre 
of $f$. This class can be used directly to give a definition of the virtual motive in the following case.
\begin{D}{\cite{BBS}} Let $f:M\to\CC$ be a regular function on the a smooth quasi-projective variety, with
 $Z= \{ df= 0\} \subset M$ the scheme theoretic degeneracy locus of $f$. The virtual motive of $Z$ is defined,
\[ [Z]_{\textrm{vir}} = -\mathbb{L}^{-\frac{\dim M}{2}} [\varphi_f] \in \mathcal{M}_{\CC} .\]
\end{D}
This really depends not only on the scheme $Z$ but on the function $f$ and space $M$. However taking the Euler characteristic 
will give a number intrinsic to $Z$. In particular if $Z$ is a moduli space of sheaves on a Calabi-Yau threefold then 
$\chi([Z]_{\textrm{vir}})$ will equal the Donaldson-Thomas invariant, by Behrend's description of the constructible function 
$\nu$ in terms of the Milnor fibre \cite{behrend}.
\paragraph{}
Under some simplifying assumptions there exists an explicit description of the virtual motive.
\begin{D}{(Property A.)}\label{A} Let $f:M\to \CC$ be a regular function on a smooth quasi-projective variety $M$. Then $f$ satisfies Property A if there exist an algebraic torus $T$ acting on $M$, so that for all $m\in M$ and $t\in T$ we have 
\[ f (t \cdot m) = \chi (t) \cdot f (m), \]
with $\chi :T \to \CC^*$ a primitive character.
\end{D}
Property A means that the map is a trivial fibration over $\CC^*$. This is easy to see since the primitivity of the character $\chi$ means there exists a one dimensional subtorus $\CC^* \subset T$ such that $\chi$ is an isomorphism when restricted to $\CC^*$.  Letting 
$M_1 = f^{-1}(1)$ and $M_0 = f^{-1}(0)$ the trivialization is
\[ \xymatrix{  \CC^*\times M_1  \ar[r]^{t\cdot m} & M \setminus  M_0 } \]  
sending a point $m$ in the fibre over $1$ to the point $t\cdot m \in  M \setminus  M_0  $.

\begin{D}{(Property B.)}\label{B} Let $f:M\to \CC$ be a regular function on a smooth quasi-projective variety $M$ with $T$ action satisfying Property A. Then $f$ satisfies Property B if there exist a one dimensional subtorus $\CC^*\subset T$, with compact fixed point set $M^{\CC^*}$ and for all $m\in M$, $\lim_{\lambda\to 0}\lambda \cdot m$ exists. 
\end{D}
A torus action with Property B is said to be circle compact. For such actions the virtual motive has a simple description.
\begin{prop}{  \cite[Behrend, Bryan, Szendr\H{o}i]{BBS} }\label{fibdiff}
Let $f:M\to \CC$ be a regular function on a smooth quasi-projective variety $M$ with $T$ action satisfying Property B. 
The motivic class of vanishing cycles $[\varphi_f]$ can be expressed as the motivic difference of the general and central fibres 
\[ [\varphi_f ] = [f^{-1}(1)] - [f^{-1}(0)] \in K_0(Var_{\CC}) \] 
in particular 
\[[Z]_{\textrm{vir} } = \LL^{-\frac{\dim M}{2}}( [f^{-1}(0)] - [f^{-1}(1)]  )\in \mathcal{M}_{\CC}.\]
\end{prop}

\section{Quantized Poisson Algebra.}\label{secpos}

Our results are best couched in the context of a quantized Poisson algebra. Kontsevich and Soibelman 
see this arising naturally in their work on motivic DT invariants of a three dimensional Calabi-Yau categories \cite{mot3CY} 
and again in the quiver context \cite{COHA}.
\begin{D} Let $Q$ be a finite quiver with pairing $\langle , \rangle_Q$. The algebra $S_Q$, is an $\mathcal{M}^{St}_\CC$-module
\[ S_Q := \prod_{\mathbf{v} \in \ZZ_{\geq 0}^I} \mathcal{M}^{St}_\CC t^{\mathbf{v}},\]
with a non-commutative product,
\[ t^{\mathbf{v}}\ast t^{\mathbf{w}} = \LL^{-\langle\mathbf{w} ,\mathbf{v} \rangle_Q}t^{\mathbf{v}+ \mathbf{w}}.\]
\end{D}
Roughly speaking this is a non-commutative ring of power series with motivic classes as coefficients. 
All the motivic classes we are interested in can be neatly packaged in two such series. The DT series,
\[\DT^{Q,W}_{\mathbf{f}}(\mathbf{t}) := \sum_{\mathbf{w}\in \ZZ^{I}_{\geq 0}} \LL^{-\frac{1}{2}\langle \mathbf{w}, \mathbf{w} \rangle_Q}
[\DT^{Q,W}(\mathbf{w},\mathbf{f})]_{\vir}\mathbf{t}^{\mathbf{w}} \in S_Q,  \]
and the reduced series,
\[\textrm{R}^{Q,W}(\mathbf{t}) := \sum_{\mathbf{v}\in \ZZ^{I}_{\geq 0}} \frac{[R^{Q,W}(\mathbf{v})]}{[\GL(\mathbf{v})]}
\mathbf{t}^{\mathbf{v}} \in S_Q.\]
In fact the coefficents of the DT series are classes in $\mathcal{M}_\CC \subset \mathcal{M}^{St}_\CC $.

\section{Results.}\label{secthms}

Here we extend the results of \cite[\S 2.7]{BBS} to a general quiver and superpotential with a linear factor. 
\paragraph{}First we work without stability and consider the Chern-Simons functional
\[ \textrm{Tr}\tilde{W}_\mathbf{v} : M^Q(\mathbf{v}) \to \CC .\]
One immediate consequence of the linearity of the superpotential is that the above map satisfies Property A, hence is a trivial fibration over $\CC^*$. To see this just consider $\CC^*$ acting diagonally on the left of the matrices appearing in the linear factor $L$, one of which will be non-zero if the trace map is. Recall the definitions of Section \ref{secqui}
\begin{eqnarray*}
M_1^{Q,W}(\mathbf{v})&:=& \Tr\tilde{W}_{\mathbf{v}}^{-1}(1), \\
M_0^{Q,W}(\mathbf{v})&:=& \Tr\tilde{W}_{\mathbf{v}}^{-1}(0), \\
R^{Q,W}(\mathbf{v})  &:=& \{ (M_\alpha)\in M(\mathbf{v})\mid R(M_{\alpha}) = 0 \},
\end{eqnarray*}
the first space here is the general fibre, the second the central fibre and the space $R^{Q,W}(\mathbf{v})$ we call the reduced space, since it is the locus of the reduced equation $R=0$. The corresponding stacks defined in Section \ref{secqui} were denoted $\mathfrak{M}_1^{Q,W}(\mathbf{v}),\mathfrak{M}_0^{Q,W}(\mathbf{v})$ and $\mathfrak{R}^{Q,W}(\mathbf{v}) $. The following proposition expresses the difference of the general and central fibres in terms of the reduced space.

\begin{prop}\label{dr}
Let $Q$ be a finite quiver with superpotential $W$. Suppose $W$ has a linear factor. For any dimension vector $\mathbf{v}\in \mathbb{Z}_{\geq 0}^I$ 
let $\mathfrak{M}_{1}^{Q,W}(\mathbf{v})$, $\mathfrak{M}_{1}^{Q,W}(\mathbf{v})$, $\mathfrak{R}^{Q,W}(\mathbf{v})$ be the stacks defined above. 
Then
\[ [\mathfrak{M}_{0}^{Q,W}(\mathbf{v})] - [\mathfrak{M}_{1}^{Q,W}(\mathbf{v})] = [\mathfrak{R}^{Q,W}(\mathbf{v})] \]
in the Grothendieck ring of stacks $K_0(St)$.
\end{prop}

\begin{proof}

The idea is to use the triviality of the fibration and the linearity of the superpotential to obtain two simple relations between the motivic classes.

\paragraph{}Fix a dimension vector $\mathbf{v}\in \ZZ^I_{\geq 0}$. Consider the map $\Tr\tilde{W}_{\mathbf{v}} : M^Q(\mathbf{v})\to \CC $, 
we stratify the base $\CC = \CC^*\sqcup \{0\}$. The fibre over $\{0\}$ we call $M^{Q,W}_0(\mathbf{v})$, the central fibre.
Since the map is a torus family with Property A, then over $\CC^*$ it is a trivial fibration, with general fibre $M_1^{Q,W}(\mathbf{v})$.
This decomposition of $M^Q(\mathbf{v})$ into two pieces gives our first motivic relation,
\begin{eqnarray}\label{1st} [M^Q(\mathbf{v})] = [M_0^{Q,W}(\mathbf{v})] + (\LL -1)[M_1^{Q,W}(\mathbf{v})]. \end{eqnarray}

\paragraph{} Now split the arrows of the quiver into two sets, if and only if they occur in the linear factor of the superpotential;
\[ A:= \{ \alpha :i\to j \mid \alpha \in L \} \textrm{ and } B:= \{\beta :i\to j \mid \beta \not\in L \} \]

Let $m = (M_\alpha)_{\alpha:i\to j}$ be a $Q$ representation of dimension vector $\mathbf{v}$, using the splitting we decompose $m$ 
into parts
\[ (m_A,m_B) \in \bigoplus_{\alpha : i \to j \in A} \Hom(\CC^{v_i},\CC^{v_j}) \times \bigoplus_{\beta : i \to j \in B} \Hom(\CC^{v_i},\CC^{v_j}).  \]
Then
\[ L(m) = L(m_A) \in \bigoplus_{\alpha : i \to j \in A} \Hom(\CC^{v_i},\CC^{v_j}), \]
and since $W=L\cdot R$ is a sum of cycles
\[ R(m) = R(m_B) \in \bigoplus_{\alpha : i \to j \in A} \Hom(\CC^{v_j},\CC^{v_i}). \]
The trace map gives a non-degenerate pairing between these spaces,
\begin{eqnarray*} 
\Tr  &:& \bigoplus_{\alpha : i \to j \in A} \Hom(\CC^{v_i},\CC^{v_j}) \times 
\bigoplus_{\alpha : i \to j \in A} \Hom(\CC^{v_j},\CC^{v_i}) \to \mathbb{C} \\
                           &:& (l,r) \mapsto \Tr (l\cdot r).
\end{eqnarray*}

We use this to compute $M_0^{Q,W}(\mathbf{v})$ in order to get a second relation. 
\paragraph{}
Let $m = (m_A,m_B)\in M_0^{Q,W}(\mathbf{v})$, so that $\Tr (L(m_A)\cdot R(m_B)) = 0$. There are two cases to consider 
firstly when $R(m_B) = 0$ and secondly when $R(m_B) \neq 0$. By definition the locus where $R=0$ is equal to $R^{Q,W}(\mathbf{v})$.
On the compliment of this set consider the projection
\begin{eqnarray*} 
\pi &:& M^{Q}(\mathbf{v})\setminus R^{Q,W}(\mathbf{v}) \to \{m_B \mid R(m_B)\neq 0 \} \\
    &:& (m_A,m_B) \mapsto m_B.
\end{eqnarray*}
This map is a trivial vector bundle. For fixed $m_B$ the condition $\Tr(L(m_A)\cdot R(m_B))=0$ is a single linear condition on the matricies 
$m_A$ in the fibre, and so in the second case the locus of $m\in M^{Q,W}_0(\mathbf{v})$ such that $R\neq 0$, is a vector bundle of 
rank one lower. Both cases considered we have a formula for the class of $M^{Q,W}_0(\mathbf{v})$,
\begin{eqnarray}\label{2nd} [M_0^{Q,W}(\mathbf{v})] = [R^{Q,W}(\mathbf{v})] + ([{M}^{Q}(\mathbf{v})]-[R^{Q,W}(\mathbf{v})])\LL^{-1}. \end{eqnarray}

Finally relations (\ref{1st}),(\ref{2nd}) together imply
\[[M_{0}^{Q,W}(\mathbf{v})] - [M_{1}^{Q,W}(\mathbf{v})] = [R^{Q,W}(\mathbf{v})]\]
Dividing by the automorphism group $\GL(\mathbf{v})$, gives the corresponding result for stacks.  

\end{proof}

 In Section \ref{secstab} we showed that for a choice of framing vector $\mathbf{f}\in \mathbb{Z}^I_{\geq 0 }\setminus \{0\}$ there exists a fine moduli space of quiver representations with Chern-Simons functional,
 \[ \textrm{Tr}\tilde{W}_{\mathbf{v}} : N^Q(\mathbf{v},\mathbf{f})\to \CC. \]
We now assume that the family satisfies Property B (circle compact). Then by Proposition \ref{fibdiff} the virtual motive of $\DT^{Q,W}(\mathbf{v},\mathbf{f})$ is the difference of the general and central fibres of $\Tr \tilde{W}_{\mathbf{v}}$ ,
 \[ [\DT^{Q,W}(\mathbf{v},\mathbf{f})]_{\textrm{vir}} = \mathbb{L}^{-\frac{\dim N^Q(\mathbf{v},\mathbf{f})}{2}} ([\Tr \tilde{W}_{\mathbf{v}}^{-1}(0)]-[\Tr \tilde{W}_{\mathbf{v}}^{-1}(1)]). \]
In Proposition \ref{dr} we ignored stability and expressed the difference of the general and central fibres as equal to the reduced space. Now on adding the stability condition we will get a recursion relation for the virtual motives in terms of the reduced spaces.

\begin{theorem}\label{stab}
Let $Q$ be a finite quiver with superpotential $W$. Suppose that $W$ has a linear factor and torus action with Property B. For any 
dimension vector $\mathbf{v} \in \mathbb{Z}_{\geq 0}^I$ and framing $\mathbf{f} \in \mathbb{Z}_{\geq 0}^I \setminus \{0\}$, 
the virtual motives of the moduli spaces $\DT^{Q,W}(\mathbf{v},\mathbf{f})$ satisfy a recursion,
\[ [\mathfrak{R}^{Q,W}(\mathbf{v})] \mathbb{L}^{\frac{\mathbf{f} \cdot \mathbf{v} }{2} }
= \sum_{\mathbf{w}\leq \mathbf{v}} 
\mathbb{L}^{-\langle \mathbf{v}-\mathbf{w},\mathbf{w}\rangle_{Q}  -1/2\langle \mathbf{w},\mathbf{w} \rangle_{Q} }
\cdot [\DT^{Q,W}(\mathbf{w},\mathbf{f})]_{\textrm{vir}} 
\cdot [\mathfrak{R}^{Q,W}(\mathbf{\mathbf{v} -\mathbf{w}})] \mathbb{L}^{-\frac{ \mathbf{f} \cdot (\mathbf{v} -\mathbf{w}) }{2}} \]
in the ring $\mathcal{M}^{St}_{\mathbb{C}}$. Or equivalently rephrased in the language of Section \ref{secpos}
\[ \R^{Q,W}(\LL^{\frac{\mathbf{f}}{2}}\mathbf{t}) = \DT^{Q,W}_\mathbf{f}(\mathbf{t}) \ast \R^{Q,W}(\LL^{-\frac{\mathbf{f}}{2}}\mathbf{t}) .\]
This uniquely determines the invariants $[\DT^{Q,W}(\mathbf{v},\mathbf{f})]_{\textrm{vir}}$ from 
the classes $[\mathfrak{R}^{Q,W}(\mathbf{v})]$.
\end{theorem}
\begin{proof}
Fix a dimension vector $\mathbf{v} \in \ZZ^I_{\geq 0}$ and a framing vector $\mathbf{f} \in \ZZ_{\geq 0}^I \setminus \{ 0\}$. First define all the objects without stability
\begin{eqnarray*}
X^Q(\mathbf{v},\mathbf{f}) &:=& M^Q(\mathbf{v}) \times \prod_{i}(\CC^{v_i})^{f_i} \\
Y^{Q,W}(\mathbf{v},\mathbf{f}) &:=& X^Q(\mathbf{v},\mathbf{f}) \cap \Tr \tilde{W}_{\mathbf{v}}^{-1} (0) \\
Z^{Q,W}(\mathbf{v},\mathbf{f}) &:=& X^Q(\mathbf{v},\mathbf{f}) \cap \Tr \tilde{W}_{\mathbf{v}}^{-1} (1) 
\end{eqnarray*}
As shown in Proposition \ref{dr} the two fibres above are related to the reduced space,
\[  [{R}^{Q,W}(\mathbf{v})]\cdot \LL^{\langle \mathbf{f} , \mathbf{v} \rangle }= [Y^{Q,W}(\mathbf{v},\mathbf{f})] - [Z^{Q,W}(\mathbf{v},\mathbf{f}) ]. \]
The stability condition introduced in Section \ref{secstab} depends upon the the span of the vectors $\{m_l\}$ under the matrices $(M_\alpha)_{\alpha : i\to j}$. This vector space was earlier defined as
\[ \CC\langle M_\alpha \rangle \{ m_l \}. \]

For a given dimension vector $\mathbf{w}\in \ZZ^I_{\geq 0}$ set, 
\begin{eqnarray*}
X^Q(\mathbf{v},\mathbf{w},\mathbf{f}) &:=& \{ (M_\alpha , m_l) \mid \underline{\dim} \CC\langle M_\alpha \rangle \{m_l\} = \mathbf{w} \} \subset X^Q(\mathbf{v},\mathbf{f})  \\
Y^{Q,W}(\mathbf{v},\mathbf{w},\mathbf{f}) &:=& X^Q(\mathbf{v},\mathbf{w},\mathbf{f}) \cap \Tr \tilde{W}_{\mathbf{v}}^{-1} (0) \subset Y^{Q,W}(\mathbf{v},\mathbf{f})  \\
Z^{Q,W}(\mathbf{v},\mathbf{w},\mathbf{f}) &:=& X^Q(\mathbf{v},\mathbf{w},\mathbf{f}) \cap \Tr \tilde{W}_{\mathbf{v}}^{-1} (1) \subset Z^{Q,W}(\mathbf{v},\mathbf{f})  .
\end{eqnarray*}
Since the stability condition required the vectors $\{m_l\}$ generate the entire representation, then in the notation of Section \ref{secstab}
\[ U^Q(\mathbf{v},\mathbf{f}) = X^Q(\mathbf{v},\mathbf{v},\mathbf{f}) \textrm{ and } N^Q(\mathbf{v},\mathbf{f}) = X^Q(\mathbf{v},\mathbf{v}, \mathbf{f}) / \GL(\mathbf{v}). \]
So the virtual motive of $\DT^{Q,W} (\mathbf{v},\mathbf{f}) $ is the difference of the general and central fibres 
 \[ [\DT^{Q,W} (\mathbf{v},\mathbf{f})]_{\textrm{vir}} = \LL^{-\frac{\dim N^Q (\mathbf{v},\mathbf{f})}{2}}\left( \frac{[Y^{Q,W}(\mathbf{v},\mathbf{v},\mathbf{f})  ]}{[\GL(\mathbf{v})]} -\frac{[Z^{Q,W}(\mathbf{v},\mathbf{v},\mathbf{f})  ]}{[\GL(\mathbf{v})]} \right) .\]
The dimension
\begin{eqnarray*}
\dim N^Q (\mathbf{v},\mathbf{f}) &=& \dim M^Q(\mathbf{v}) - \dim \GL(\mathbf{v}) + \dim \prod_{i\in I}  (\CC^{v_i})^{f_i}  \\
 &=& - \langle  \mathbf{v},\mathbf{v} \rangle_Q +  \mathbf{v} \cdot \mathbf{f}   .
\end{eqnarray*}
The remaining task is to compute the difference $[Y^{Q,W}(\mathbf{v},\mathbf{v},\mathbf{f})]-[Z^{Q,W}(\mathbf{v},\mathbf{v},\mathbf{f})]$. Let us start with 
$Y^{Q,W}(\mathbf{v},\mathbf{w},\mathbf{f})$ and $Z^{Q,W}(\mathbf{v},\mathbf{w},\mathbf{f})$.
\paragraph{} Let $Gr(\mathbf{w},\mathbf{v})$ be the Grassmannian of $\mathbf{w}$ dimensional subspaces in $\mathbf{v}$ dimensional space, that is 
$Gr(\mathbf{w},\mathbf{v}) = \prod_{i\in I}Gr(w_i,v_i)$. An element of $Y^{Q,W}(\mathbf{v},\mathbf{w},\mathbf{f})$ defines a subspace 
$\CC\langle M_\alpha \rangle \{m_l\}$ with dimension vector $\mathbf{w} \in \ZZ_{\geq 0}^I$, the associated map
\[ Y^{Q,W}(\mathbf{v},\mathbf{w},\mathbf{f}) \to Gr(\mathbf{w},\mathbf{v}) \] is a Zariski fibration. To compute the motivic class of the fibre we fix a basis 
so that 
\[ M_{\alpha}  = \left( \begin{array}{cc} M'_\alpha & M^*_\alpha \\ 0 & M''_\alpha \end{array} \right) \in \left( \begin{array}{cc} \Hom(w_i,w_j) & \Hom(w_i, v_j-w_j) \\ 0 & \Hom (v_i-w_i,v_j-w_j) \end{array} \right)  .\]
where $\alpha : i \to j $, and vectors 
\[ m_l = \left( \begin{array}{c} m_l' \\ 0 \end{array} \right) \in \left( \begin{array}{c} \Hom(1,w_i)  \\ 0 \end{array} \right) \]
where $m_l$ is at vertex $i\in I$. The image of the vectors $m'_l$ under the matrices $M'_\alpha$ is now the entire $\mathbf{w}$ dimensional subspace. The
 Chern-Simons functional also splits with respect to this basis,
\[ \Tr \tilde{W}_{\mathbf{v}} (M_\alpha)  = \Tr \tilde{W}_{\mathbf{w}} (M'_\alpha) + \Tr \tilde{W}_{\mathbf{v}-\mathbf{w}} (M''_\alpha) = 0, \]
in particular there is no restriction on the $M^*_\alpha$, and they factor out an affine space of dimension 
$-\langle \mathbf{v}-\mathbf{w}, \mathbf{w} \rangle_Q +( \mathbf{v}-\mathbf{w})\cdot \mathbf{w}  $. The two cases to consider are
\[ \{ \Tr \tilde{W}_{\mathbf{w}} (M'_\alpha) = \Tr \tilde{W}_{\mathbf{v}-\mathbf{w}} (M''_\alpha) = 0 \}, \]
and
\[ \{ \Tr \tilde{W}_{\mathbf{w}} (M'_\alpha) = -\Tr\tilde{W}_{\mathbf{v}-\mathbf{w}} (M''_\alpha) \neq 0 \}. \]
In the first case we get an element of $Y^{Q,W}(\mathbf{w},\mathbf{w},\mathbf{f})$ and an element of $Y^{Q,W}(\mathbf{v}-\mathbf{w})$. The other stratum is a 
trivial $\CC^*$ bundle, by the nonzero value of the Chern-Simons functional. Looking at the fibre over $1$ gives an element of 
$Z^{Q,W}(\mathbf{w},\mathbf{w},\mathbf{f})$ and an element of $Z^{Q,W}(\mathbf{v}-\mathbf{w})$. The total motivic class of the fibre is then,
\begin{eqnarray*}\lefteqn{ [Y^{Q,W}(\mathbf{w},\mathbf{w},\mathbf{f}) ]\cdot \LL^{-\langle \mathbf{v}-\mathbf{w}  ,\mathbf{w}  \rangle_Q + (\mathbf{v}-\mathbf{w})\cdot  \mathbf{w}} \cdot [Y^{Q,W}(\mathbf{v}-\mathbf{w},\mathbf{f})] \LL^{ -  \mathbf{f} \cdot (\mathbf{v}-\mathbf{w} )} }  \\
&& +(\LL  -1)  [Z^{Q,W}(\mathbf{w},\mathbf{w},\mathbf{f}) ] \cdot \LL^{ -\langle \mathbf{v}-\mathbf{w}  ,\mathbf{w}  \rangle_Q + \langle \mathbf{v}-\mathbf{w},  \mathbf{w} \rangle} \cdot [Z^{Q,W}(\mathbf{v}-\mathbf{w},\mathbf{f})] \LL^{-   \mathbf{f} \cdot (\mathbf{v}-\mathbf{w} ) } .
\end{eqnarray*}
The space $Z^{Q,W}(\mathbf{v},\mathbf{w},\mathbf{f})$ also fibres over the Grassmannian $Gr(\mathbf{w},\mathbf{v})$, the motivic class of the fibre is 
computed similarly, as
\begin{eqnarray*}\lefteqn{ [Y^{Q,W}(\mathbf{w},\mathbf{w},\mathbf{f}) ]\cdot \LL^{-\langle \mathbf{v}-\mathbf{w}  ,\mathbf{w}  \rangle_Q +  (\mathbf{v}-\mathbf{w})\cdot  \mathbf{w} } \cdot [Z^{Q,W}(\mathbf{v}-\mathbf{w},\mathbf{f})] \LL^{ - \mathbf{f} \cdot (\mathbf{v}-\mathbf{w} ) } }  \\
&& +(\LL  -2)  [Z^{Q,W}(\mathbf{w},\mathbf{w},\mathbf{f}) ] \cdot \LL^{ -\langle \mathbf{v}-\mathbf{w}  ,\mathbf{w}  \rangle_Q + (\mathbf{v}-\mathbf{w})\cdot  \mathbf{w} } \cdot [Z^{Q,W}(\mathbf{v}-\mathbf{w},\mathbf{f})] \LL^{- \mathbf{f} \cdot (\mathbf{v}-\mathbf{w} ) } \\
&& +[Z^{Q,W}(\mathbf{w},\mathbf{w},\mathbf{f}) ] \cdot \LL^{ -\langle \mathbf{v}-\mathbf{w}  ,\mathbf{w}  \rangle_Q + (\mathbf{v}-\mathbf{w})\cdot  \mathbf{w} } \cdot [Y^{Q,W}(\mathbf{v}-\mathbf{w},\mathbf{f})] \LL^{- \mathbf{f} \cdot (\mathbf{v}-\mathbf{w} ) }.
\end{eqnarray*}
We are now ready to deduce the recursion. Stratifying $Y^{Q,W}(\mathbf{v},\mathbf{f})$ and $Y^{Q,W}(\mathbf{v},\mathbf{f})$ by the dimension of 
$\CC\langle M_\alpha \rangle \{m_l\}$ we have
\[ Y^{Q,W}(\mathbf{v},\mathbf{f}) = \coprod_{\mathbf{w}\leq \mathbf{v}} Y^{Q,W}(\mathbf{v},\mathbf{w},\mathbf{f}) 
 \textrm{  and  } Z^{Q,W}(\mathbf{v},\mathbf{f}) = \coprod_{\mathbf{w}\leq \mathbf{v}} Z^{Q,W}(\mathbf{v},\mathbf{w},\mathbf{f}). \]
As mentioned the motivic difference of these two spaces was equal to the class of the reduced space (Proposition \ref{dr})
\[ [R^{Q,W} (\mathbf{v}) ] \LL^{ \mathbf{f}\cdot \mathbf{v} } = \sum_{\mathbf{w}\leq \mathbf{v}} [ Y^{Q,W}(\mathbf{v},\mathbf{w},\mathbf{f}) ] - [Z^{Q,W}(\mathbf{v},\mathbf{w},\mathbf{f})] .\]
Substituting in our formulas for $Y^{Q,W}(\mathbf{v},\mathbf{w},\mathbf{f})$ and $Z^{Q,W}(\mathbf{v},\mathbf{w},\mathbf{f})$ above gives
\begin{eqnarray*}
[R^{Q,W} (\mathbf{v}) ] \LL^{ \mathbf{f}\cdot \mathbf{v} } & = & \sum_{\mathbf{w}\leq \mathbf{v}} [Gr(\mathbf{w},\mathbf{v})] 
 \LL^{-\langle \mathbf{v}-\mathbf{w}  ,\mathbf{w}  \rangle_Q + ( \mathbf{v}-\mathbf{w})\cdot \mathbf{w}  -  \mathbf{f} \cdot (\mathbf{v}-\mathbf{w}) } \\
&&\cdot
\left([Y^{Q,W}(\mathbf{w},\mathbf{w},\mathbf{f}) ]  - [Z^{Q,W}(\mathbf{w},\mathbf{w},\mathbf{f}) ]  \right) \\
&  &\cdot  \left( [ Y^{Q,W}(\mathbf{v}-\mathbf{w},\mathbf{f}) ] -[Z^{Q,W}(\mathbf{v}-\mathbf{w},\mathbf{f}) ]  \right) \\
&=&  \sum_{\mathbf{w}\leq \mathbf{v}}  \frac{[\GL(\mathbf{v})]}{[\GL(\mathbf{w})][\GL(\mathbf{v}-\mathbf{w})]} \LL^{-\langle \mathbf{v} - \mathbf{w} ,  \mathbf{w}\rangle_Q} \\
& & \cdot \left([Y^{Q,W}(\mathbf{w},\mathbf{w},\mathbf{f}) ]  - [Z^{Q,W}(\mathbf{w},\mathbf{w},\mathbf{f}) ]  \right) \cdot [R^{Q,W} (\mathbf{v}-\mathbf{w}) ] \\
&=& \sum_{\mathbf{w}\leq \mathbf{v}}  \frac{[\GL(\mathbf{v})]}{[\GL(\mathbf{v}-\mathbf{w})]} \LL^{ -\langle \mathbf{v}-\mathbf{w},\mathbf{w}\rangle_{Q}  -1/2\langle \mathbf{w},\mathbf{w} \rangle_{Q}+ 1/2 \mathbf{f}\cdot \mathbf{w} }\\
& &
\cdot [\DT^{Q,W}(\mathbf{w},\mathbf{f})]_{\textrm{vir}} 
\cdot [{R}^{Q,W}(\mathbf{\mathbf{v} -\mathbf{w}})]
\end{eqnarray*}
Finally dividing out by the automorphism groups $\GL(\mathbf{v}) \textrm{ and }\GL(\mathbf{v}-\mathbf{w})$ gives,
\[ [\mathfrak{R}^{Q,W}(\mathbf{v})] \mathbb{L}^{\frac{ \mathbf{f} \cdot \mathbf{v} }{2} }
= \sum_{\mathbf{w}\leq \mathbf{v}} 
\mathbb{L}^{-\langle \mathbf{v}-\mathbf{w},\mathbf{w}\rangle_{Q}  -1/2\langle \mathbf{w},\mathbf{w} \rangle_{Q} }
\cdot [\DT^{Q,W}(\mathbf{w},\mathbf{f})]_{\textrm{vir}} 
\cdot [\mathfrak{R}^{Q,W}(\mathbf{\mathbf{v} -\mathbf{w}})] \mathbb{L}^{-\frac{ \mathbf{f} \cdot (\mathbf{v} -\mathbf{w}) }{2}}. \]
\end{proof}

\section{A Geometric Application.}\label{}

We outline an application of dimensional reduction, to computing motivic DT invariants of orbifolds.
\paragraph{}
Let $\ZZ_n$ be the finite group generated by $\zeta = e^{2\pi i / n}$. This acts on $\CC^3$ sending $(x,y,z) \mapsto (x,\zeta y, \zeta ^{-1}z)$. 
The quotient can be considered as a smooth Deligne-Mumford stack (orbifold),
\[\mathcal{X} = [\CC^3 / \ZZ_n] . \]
Irreducible representations of $\ZZ_n$ are classified by an integer $\{0,1,\ldots,n-1\}$ and so it follows that $\ZZ_{\geq 0}^n$ 
indexes all representations of $\ZZ_n$. The Hilbert scheme of degree zero substacks of $\mathcal{X}$ with class 
$\mathbf{v} \in \ZZ_{\geq 0}^n$ is defined concretely by,
\[ \Hilb^{\mathbf{v}}(\mathcal{X}) = \{ Z\subset \CC^3 \mid H^0(\mathcal{O}_Z) \iso \mathbf{v} \textrm{ as a } \ZZ_n \textrm{ rep.}\}.\]
This is a $\ZZ_n$ invariant component of the Hilbert scheme on $\CC^3$. There consequently is a similar quiver description for the above 
Hilbert schemes, where the $\ZZ_n$ invariance relates the quiver to an affine Dynkin diagram of type $A_n$.
\paragraph{}
Applying the dimensional reduction argument, the reduced spaces become certain commuting varieties whose motives we calculate. 
In summary if we let,
\[\DT_\mathcal{X}(\mathbf{t}) = \sum_{\mathbf{v}\in \ZZ_{\geq 0}^{n}} [\Hilb^\mathbf{v}(\mathcal{X})]_{\vir} 
t_0^{v_0}\cdots t_{n-1}^{v_{n-1}}, \]
then it can be shown \cite{a_n},

\begin{eqnarray*}
 \DT_\mathcal{X}(\mathbf{t}) & = &  \prod_{m\geq 1}\prod_{k=1}^{m} 
\frac{1}{1-\mathbb{L}^{2-k+\frac{m}{2}}t^m}
\frac{1}{(1-\mathbb{L}^{1-k+\frac{m}{2}}t^m)^{n-1}}  \\
&  & \cdot \prod_{m\geq 1} \prod_{k=1}^{m} \prod_{0< a \leq b< n} 
\frac{ 1 }{ 1-\mathbb{L}^{1-k+\frac{m}{2}} t^m t_{[a,b]} }
\frac{ 1 }{ 1-\mathbb{L}^{1-k+\frac{m}{2}} t^m t_{[a,b]}^{-1} }.
\end{eqnarray*}
Where $ t = t_0 \cdot t_1 \cdots t_{n-1} $, $ t_{[a,b]} = t_a \cdot t_{a+1} \cdots t_{b} $ and the above product is taken 
in the standard ring of commuting indeterminates. 
This provides a motivic refinement of Ben Young's results on enumerating three dimensional colored partitions \cite{young}.

\section{Acknowledgements.}\label{}

Many thanks to my supervisor Jim Bryan. I am partially supported by a Four Year Doctoral Fellowship (4YF), University of British Columbia.

\end{document}